\documentclass{amsart}
\usepackage[T1]{fontenc}
\usepackage[utf8]{inputenc}
\usepackage{enumerate}
\usepackage{amsthm}
\usepackage{amscd}
\def\car{{\mathbf 1}}

\def\/{\, | \,}
\def\e{\varepsilon}

\def\C{C}
\def\I{{\mathcal I}}

\newcommand{\Lip}{\operatorname{Lip}}

\newcommand{\Rem}{\operatorname{Rem}}
\newcommand{\Lin}{\operatorname{Lin}}

\def\D{{\mathbb D}}

\def\T{{\mathcal T}}

\def\R{{\mathbf R}}
\def\L{{\mathbf L}}

\def\Proba0{{\mathcal P}_0}
\def\1e{\frac{1}{\e}}

\def\L{{\mathcal L}}

\def\e{\varepsilon}
\def\/{\, | \,}

\def\d{\text{ d}}

\def\car{{\mathbf 1}}
\def\I{{\mathcal I}}

\def\car{{\mathbf 1}}

\def\P{{\text P}}

\def\D{{\mathbb D}}

\def\<{\langle}
\def\>{\rangle}
\def\({\Bigl(}
\def\){\Bigr)}

\def\D{{\mathbb D}}
\def\CC{{\mathcal C}}

\def\<<{\langle\!\langle\,}
\def\>>{\,\rangle\!\rangle}
\def\embed{{\mathfrak J}_\beta}

\newcommand{\Id}{\operatorname{Id}}

\newcommand{\Hol}{\operatorname{Hol}}

\newcommand{\Her}{{\mathcal H}}
\def\l2{l^2}
\def\hatH{{\mathcal K}}
\def\d{\, \text{d}}
\newcommand{\trace}{\operatorname{trace}}
\newcommand{\N}{{\mathbf N}}

\newcommand{\thmref}[1]{Theorem [\ref{#1}]}

\newtheorem{theorem}{Theorem}[section]

\newtheorem{lemma}{Lemma}[section]
\newtheorem{defn}{Definition}

\newtheorem*{example*}{Example}
\def\esp#1#2{\mathbb{E}_{#1}\left[ #2 \right]}
\begin{document}
\title{Higher order expansions via Stein method} \author{L. Coutin}
\address{Institute  of Mathematics\\
  Universit\'e Toulouse 3\\
  Toulouse, France} \email{laure.coutin@math.univ-toulouse.fr}
\author{L. Decreusefond}
\address{Institut Telecom, Telecom ParisTech, CNRS LTCI\\
  Paris, France} \email{Laurent.Decreusefond@telecom-paristech.fr}
\thanks{Both authors were partially supported by ANR-10-BLAN-0121.}
\begin{abstract}
  This paper is a sequel of \cite{CD:2012}. We show how to establish a
  functional Edgeworth expansion of any order thanks to the Stein
  method. We apply the procedure to the Brownian approximation of
  compensated Poisson process and to the linear interpolation of the
  Brownian motion. It is then apparent that these two expansions are
  of rather different form.
\end{abstract}
\maketitle{}

\section{Introduction}
\label{sec:introduction}
For $(\mu_n,\, n\ge 1)$, a sequence of probability measures which
satisfies a central limit theorem, i.e. $\mu_n$ converges weakly to a
Gaussian measure, it may be natural to ponder how this limit could be
refined. That means, can we find an alternative distribution $\mu$ so
that the speed of convergence of $\mu_n$ towards $\mu$ is faster than
the convergence of $\mu_n$ to the Gaussian measure of the CLT ? 
For instance, for a sequence of i.i.d. centered random variables
$(X_n,\, n\ge 1)$ with unit variance, if we consider
$S_n=n^{-1/2}\sum_{j=1}^n X_j$, a Taylor expansion of the
characteristic function of $S_n$ yields the expansion:
\begin{equation*}
  \esp{}{e^{itS_n}}=e^{-t^2/2}\left[ 1
    +\frac{(it)^3\gamma}{6\sqrt{n}}+\frac{(it)^4(\tau
      -3)}{24n}+\frac{(it)^6\gamma^2}{72n}\right] +o\left(\frac1n\right).
\end{equation*}
where $\gamma=\esp{}{X_1^3}$ and $\tau=\esp{}{X_1^4}$. This can be
interpreted as the distribution of $S_n$ to be close to the measure
with density $g_n$ given by
\begin{equation*}
  g_n(x)=\frac{e^{-x^2/2}}{\sqrt{2\pi}}\left(1+\frac{\gamma}{6\sqrt{n}}\,
    \Her_3(x)+\frac{(\tau-3)}{24
      n}\, \Her_4(x)+\frac{\gamma^2}{72n}\, \Her_6(x)\right), 
\end{equation*}
where $\Her_n$ is the $n$-th Hermite polynomial.  As the comparison of the
characteristic functions of two probability measures does not give
easily quantitative estimates regarding probability of events, moments
and so on; it is necessary to investigate alternative distances
between distribution of random variables.

One of the most natural distance is the so-called Kolmogorov distance
defined, for measures supported on $\R$, by
\begin{equation*}
  \d_{\text{Kol}}(\mu,\, \nu)=\sup_{x\in \R} \Bigl|\mu(-\infty,x]-\nu(-\infty,x]\Bigr|,
\end{equation*}
or more generally on $\R^k$,
\begin{equation*}
  \d_{\text{Kol}}(\mu,\, \nu)=\sup_{(x_1,\cdots,x_k)\in \R^k} \Bigl|\mu(\times_{j=1}^k(-\infty,x_j])-\nu(\times_{j=1}^k(-\infty,x_j])\Bigr|.
\end{equation*}
This definition is hardly usable for probability measures on more
abstract spaces like Hilbert spaces, space of continuous functions,
etc.  On the contrary, the Rubinstein distance can be defined in great
generality. Assume that $\mu$ and $\nu$ are two probability measures
on a metric space $(X,\, d)$. The space of $d$-Lipschitz functions is
the set of functions $f$ such that there exists $c>0$, depending only
on $f$, satisfying
\begin{equation*}
  |f(x)-f(y)|\le c\,  d(x,\,y), \text{ for all } x,\, y \in X.
\end{equation*}
We denote by $\Lip_1$ the set of $d$-Lipschitz functions for which $c$
can be taken equal to $1$. Then, the Rubinstein distance is defined by
\begin{equation*}
  d_R(\mu,\, \nu)=\sup_{F\in \Lip_1} \int_X F\d \mu-\int_X F\d\nu.
\end{equation*}
It is well known (see \cite{MR982264}) that if $(X,\,d)$ is separable,
$(\mu_n,\, n\ge 1)$ converges weakly to $\mu$ if and only if
$d_R(\mu_n,\, \mu)$ tends to $0$ as $n$ goes to infinity. Moreover,
according to \cite{MR2235448},
\begin{equation}\label{eq_edgeworth:1}
  d_{\text{Kol}}(\mu,\, \nu)\le 2\sqrt{d_R(\mu,\, \nu)}.
\end{equation}
The Stein method, which dates back to the seventies, is one approach
to evaluate such distances.  Since \cite{MR836279,MR714144}, it is
well known that Stein method can also lead to expansions of higher
order by pursuing the development. In a previous paper \cite{CD:2012},
we proved quantitative versions of some well known theorems: the
Donsker Theorem, convergence of Poisson processes of increasing
intensity towards a Brownian motion and approximation of a Brownian
motion by increasingly refined linear interpolations. We now want to
show that the same framework can be used to derive higher order
expansions, even in functional spaces.

Before going deeply into technicalities, let us just show how this
works on a simple $1$ dimensional example.  Imagine that we want to
precise the speed of convergence of the well-known limit in
distribution:
\begin{equation*}
  \frac{1}{\sqrt{\lambda}}(X_\lambda-\lambda) \xrightarrow{\lambda \to \infty} {\mathcal N}(0,\, 1),
\end{equation*}
where $X_\lambda$ is a Poisson random variable of parameter $\lambda$.
We consider the Rubinstein distance between the distribution of
$\tilde{X}_\lambda=\lambda^{-1/2}(X_\lambda-\lambda)$ and ${\mathcal
  N}(0,\, 1)$, which is defined as
\begin{equation}\label{eq_gaussapp20:1}
  \text{d}_(\tilde{X}_\lambda,\, {\mathcal
    N}(0,\, 1))=\sup_{F\in \Lip_1} \esp{}{F(\tilde{X}_\lambda)}-\esp{}{F({\mathcal N}(0,\, 1))}.
\end{equation}
The well known Stein Lemma stands that for any $F\in \Lip_1$, there
exists $\psi_F\in {\mathcal C}^2_b$ such that for all $x\in \R$,
\begin{equation*}
  F(x)-\esp{}{F({\mathcal N}(0,\, 1))}=x\,
  \psi_F(x)-\psi_F^\prime(x). 
\end{equation*}
\begin{equation*}
  \Vert \psi_F^\prime\|_\infty \le 1, \   \Vert \psi_F^{\prime\prime}\|_\infty \le 2.
\end{equation*}
Hence, instead of the right-hand-side of \eqref{eq_gaussapp20:1}, we
are lead to estimate
\begin{equation}\label{eq_gaussapp20:3}
  \sup_{\Vert \psi^\prime\|_\infty \le 1, \   \Vert
    \psi^{\prime\prime}\|_\infty \le 2}\esp{}{\tilde{X}_\lambda \psi(\tilde{X}_\lambda)-\psi^\prime(\tilde{X}_\lambda)}.
\end{equation}
This is where the Malliavin-Stein approach differs from the classical
line of thought. In order to transform the last expression, instead of
constructing a coupling, we resort to the integration by parts formula
for functionals of Poisson random variable. The next formula is well
known or can be viewed as a consequence of \eqref{eq_gaussapp6:4}:
\begin{equation*}
  \esp{}{\tilde{X}_\lambda G(\tilde{X}_\lambda)}=\sqrt{\lambda} \ \esp{}{G(\tilde{X}_\lambda+1/\sqrt{\lambda})-G(\tilde{X}_\lambda)}.
\end{equation*}
Hence, \eqref{eq_gaussapp20:3} is transformed into
\begin{equation}\label{eq_gaussapp20:2}
  \sup_{\Vert \psi^\prime\|_\infty \le 1, \   \Vert
    \psi^{\prime\prime}\|_\infty \le 2} \esp{}{\sqrt{\lambda}(\psi(\tilde{X}_\lambda+1/\sqrt{\lambda})-\psi(\tilde{X}_\lambda))-\psi^\prime(\tilde{X}_\lambda)}.
\end{equation}
According to the Taylor formula
\begin{equation*}
  \psi(\tilde{X}_\lambda+1/\sqrt{\lambda})-\psi(\tilde{X}_\lambda)=\frac{1}{\sqrt{\lambda}}\psi^\prime(\tilde{X}_\lambda)+\frac{1}{2\lambda}\psi^{\prime\prime}(\tilde{X}_\lambda+\theta/\sqrt{\lambda}),
\end{equation*}
where $\theta\in (0,1)$. If we plug this expansion into
\eqref{eq_gaussapp20:2}, the term containing $\psi^\prime$ is
miraculously vanishing and we are left with only the second order
term. This leads to the estimate
\begin{equation*}
  \text{d}_{R}\left( \tilde{X}_\lambda,\
    {\mathcal N}(0,\, 1)\right) \le \frac{1}{\sqrt{\lambda}}\cdotp
\end{equation*}
We now want to precise the expansion. For, we go one step further in
the Taylor formula (assuming $\psi$ has enough regularity)
\begin{equation*}
  \psi(\tilde{X}_\lambda+1/\sqrt{\lambda})-\psi(\tilde{X}_\lambda)=\frac{1}{\sqrt{\lambda}}\psi^\prime(\tilde{X}_\lambda)+\frac{1}{2\lambda}\psi^{\prime\prime}(\tilde{X}_\lambda)+\frac{1}{6\lambda^{3/2}}\psi^{(3)}(\tilde{X}+\theta/\sqrt{\lambda}).
\end{equation*}
Hence,
\begin{equation}\label{eq_gaussapp20:4}
  \esp{}{ \tilde{X}_\lambda \psi(\tilde{X}_\lambda)-\psi^\prime(\tilde{X}_\lambda)}= \frac{1}{2\sqrt{\lambda}}\esp{}{\psi^{\prime\prime}(\tilde{X}_\lambda)}+\frac{1}{6\lambda}\esp{}{\psi^{(3)}(\tilde{X}+\theta/\sqrt{\lambda})}.
\end{equation}
If $F$ is twice differentiable with bounded derivatives then $\psi_F$
is three time differentiable with bounded derivatives, hence the last
term of \eqref{eq_gaussapp20:4} is bounded by
$\lambda^{-1}\|\psi_F^{(3)}\|_\infty/6$. Moreover, the first part of
the reasoning shows that
\begin{equation*}
  \esp{}{\psi_F^{\prime\prime}(\tilde{X}_\lambda)}=
  \esp{}{\psi_F^{\prime\prime}(\mathcal N(0,\, 1))}+O(\lambda^{-1/2}).
\end{equation*}
Combining the last two results, we obtain that for $F$ twice
differentiable
\begin{multline*}
  \esp{}{F(\tilde{X}_\lambda)}-\esp{}{F({\mathcal N}(0,\, 1))}=\esp{}{
    \tilde{X}_\lambda
    \psi_F(\tilde{X}_\lambda)-\psi_F^\prime(\tilde{X}_\lambda)}\\
  =\frac{1}{2\sqrt{\lambda}}\esp{}{\psi_F^{\prime\prime}(\mathcal
    N(0,\, 1))} +O(\lambda^{-1}).
\end{multline*}
This line of thought can be pursued at any order provided that $F$ is
assumed to have sufficient regularity and we get an Edgeworth
expansion up to any power of $\lambda^{-1/2}$.  Using the properties
of Hermite polynomials, this leads to the expansion:
\begin{equation*}
  \esp{}{F(\tilde{X}_\lambda)}-\esp{}{F({\mathcal N}(0,\, 1))}
  =\frac{1}{6\sqrt{\lambda}}\esp{}{(F \Her_3)(\mathcal N(0,\, 1))}+O(\lambda^{-1}).
\end{equation*}
The paper is organized as follows. In Section
\ref{sec:gauss-appr-slob}, we recall the functional structure on which
the computations are made. In Section \ref{sec:norm-appr-poiss}, we
establish the Edgeworth expansion for the Poisson approximation of the
Brownian motion. In Section \ref{sec:linear-interpolation}, we apply
the same procedure to derive an Edgeworth expansion for the linear
approximation of the Brownian motion, which turns to be of a very
different flavor.  In \cite{CD:2012}, we computed the first order term
in the Donsker Theorem, we could as well pursue the expansion. It
would be a mixture of the two previous kinds of expansion.

\section{Gaussian structure on $\l2$}
\label{sec:gauss-appr-slob}

\subsection{Wiener measure}
\label{sec:wiener-measure}

For the three examples mentioned above, we seek to compare
quantitatively the distribution of a piecewise differentiable process
with that of a Brownian motion, hence we need to consider a functional
space to which the sample-paths of both processes belong to.  It has
been established in \cite{CD:2012} that a convenient space is the
space of $\beta$-differentiable functions for any $\beta<1/2$, which
we describe now. We refer to \cite{samko93} for details on fractional calculus. For $f\in \L^2([0,1];\ dt),$ (denoted by $\L^2$ for
short) the left and right fractional integrals of $f$ are defined by~:
\begin{align*}
  (I_{0^+}^{\alpha}f)(x) & =
  \frac{1}{\Gamma(\alpha)}\int_0^xf(t)(x-t)^{\alpha-1}\d t\ ,\ x\ge
  0,\\
  (I_{1^-}^{\alpha}f)(x) & =
  \frac{1}{\Gamma(\alpha)}\int_x^1f(t)(t-x)^{\alpha-1}\d t\ ,\ x\le 1,
\end{align*}
where $\alpha>0$ and $I^0_{0^+}=I^0_{1^-}=\Id.$ For any $\alpha\ge 0$,
any $f,\, g\in \L^2$ and $g\in \L^2$, we have~:
\begin{equation}
  \label{int_parties_frac}
  \int_0^1 f(s)(I_{0^+}^\alpha g)(s)\d s = \int_0^1 (I_{1^-}^\alpha 
  f)(s)g(s)\d s.
\end{equation}
The Besov-Liouville space $I^\alpha_{0^+}(\L^2):= \I_{\alpha,2}^+$ is
usually equipped with the norm~:
\begin{equation}
  \label{normedansIap}
  \|  I^{\alpha}_{0^+}f \| _{ \I_{\alpha,2}^+}=\| f\|_{\L^2}.
\end{equation}
Analogously, the Besov-Liouville space $I^\alpha_{1^-}(\L^2):=
\I_{\alpha,2}^-$ is usually equipped with the norm~:
\begin{equation*}
  \| I^{\alpha}_{1^-}f \| _{ \I_{\alpha,2}^-}=\|  f\|_{\L^2}.
\end{equation*}
Both spaces are Hilbert spaces included in $\L^2$ and if $(e_n,\, n\in
\N)$ denote a complete orthonormal basis of $\L^2$, then
$(k_n^\alpha:=I^\alpha_{1^-}e_n,\, n\in \N)$ is a complete orthonormal
basis of $\I_{\alpha,2}^-$. Moreover, we have the following Theorem,
proved in \cite{CD:2012}.
\begin{theorem}
  \label{thm:hilbertschmidt}
  The canonical embedding $\kappa_\alpha$ from $\I^-_{\alpha,2}$ into
  $\L^2$ is Hilbert-Schmidt if and only if $\alpha >1/2$.  Moreover,
  \begin{equation}\label{eq_gaussapp20:5}
    c_\alpha:=      \|\kappa_\alpha\|_{HS}=\| I^{\alpha}_{0^+}\|_{HS}=\|I^\alpha_{1^-}\|_{HS} = \frac{1}{2\Gamma(\alpha)}
    \left(\frac{1}{\alpha( \alpha-1/2)}\right)^{1/2}.
  \end{equation}
\end{theorem}

To construct the Wiener measure on $\I_{\beta,\, 2}$, we start from
the It\^o-Nisio theorem. Let $(X_n,\, n\ge 1)$ be a sequence of
independent centered Gaussian random variables of unit variance
defined on a common probability space $(\Omega,\, {\mathcal A},\,
\P)$.  Then,
\begin{equation*}
  B(t):= \sum_{n\ge 1}X_nI_{0^+}^1(e_n)(t)
\end{equation*}
converges almost-surely for any $t\in [0,\, 1]$. Moreover, the
convergence holds in $L^2(\Omega;\, \I_{\beta,\, 2})$, so that,
$\mu_\beta$, the Wiener measure on $\I_{\beta,2}$ is the image measure
of $\P$ by the map $B$. Thus, $\mu_\beta$ is a Gaussian measure on
$\I_{\beta,\, 2}$ of covariance operator given by
\begin{equation*}
  V_\beta=I_{0^+}^\beta\circ I_{0^+}^{1-\beta}\circ
  I_{1^-}^{1-\beta}\circ I_{0^+}^{-\beta}.
\end{equation*}
This means that
\begin{equation*}
  \esp{\mu_\beta}{\exp(i\langle \eta,\, \omega\rangle_{\I_{\beta,\,
        2}})}=\exp(-\frac 12 \langle V_\beta\eta,\,
  \eta\rangle_{\I_{\beta,\, 2}}).
\end{equation*}
We could thus in principle make all the computations in $\I_{\beta,\,
  2}$. It turns out that we were not able to be explicit in the
computations of some traces of some involved operators, the
expressions of which turned to be rather straightforward in $l^2(\N)$
(where $\N$ is the set of positive integers). This is why we transfer
all the structure to $l^2(\N)$, denoted henceforth by $\l2$ for
short. This is done at no loss of generality nor precision since there
exists a bijective isometry between $\I_{\beta,\, 2}$ and $\l2$.

Actually, the canonical isometry is given by the Fourier expansion of
the $\beta$-th \textit{derivative} of an element of $\I_{\beta,\, 2}$:
for $f\in \I_{\beta,\, 2}$, we denote by $\partial_\beta f$ the unique
element of $\L^2$ such that $f=I^\beta_{0^+}\partial_\beta f$. We
denote by $(x_n,\, n\ge 1)$ a complete orthonormal basis of $\l2$. In
what follows, we adopt the usual notations regarding scalar product on
$\l2$:
\begin{equation*}
  \| x\|^2_{\l2}= \sum_{n=1}^\infty |x_n|^2 \text{ and } x.y=\sum_{n=1}^\infty x_ny_n, \text{ for all } x,\, y \in \l2.  
\end{equation*}
For the sake of simplicity, we also denote by a dot the scalar product
in $(\l2)^{\otimes k}$ for any integer $k$.

Consider the map $\embed$ defined by:
\begin{align*}
  \embed \, :\, \I_{\beta,2} & \longrightarrow \l2\\
  f & \longmapsto \sum_{n\ge 1}\left( \int_0^1 \partial_\beta
    f(s)e_n(s)\d s\right)\ x_n.
\end{align*}
According to the properties of Gaussian measure (see
\cite{MR0461643}), we have the following result.
\begin{theorem}
  \label{thm_gaussapp10:6} Let $\mu_\beta$ denote the Wiener measure
  on $\I_{\beta,\, 2}$.  Then $\embed^*\mu_\beta=m_{\beta}$, where
  $m_\beta$ is the Gaussian measure on $\l2$ of covariance operator
  given by
  \begin{equation*}
    S_\beta=\sum_{n,m\ge 1} \left(\int_0^1 k_n^{1-\beta}(s)k^{1-\beta}_m(s)\d
      s\right)\ x_n\otimes x_m.
  \end{equation*}
\end{theorem}

\subsection{Dirichlet structure}
\label{sec:dirichlet-structure}

By $\CC^k_b(\l2;\, X)$, we denote the space of $k$-times Fréchet
differentiable functions from $\l2$ into an Hilbert space $X$ with
bounded derivatives: A function $F$ belongs to $\CC^k_b(\l2;\, X)$
whenever
\begin{equation*}
  \|F\|_{\CC^k_b(\l2;\, X)}:= \sup_{j=1,\, \cdots,\, k} \sup_{x\in \l2} \| \nabla^{(j)} F(x)\|_{X\otimes (\l2)^{\otimes j}}<\infty.
\end{equation*}
\begin{defn}
  The Ornstein-Uhlenbeck semi-group on $(\l2,\, m_\beta)$ is defined
  for any $F\in L^2(\l2, {\mathcal F},m_\beta; \, X)$ by
  \begin{align*}
    P_t^\beta F(u)&=\int_{\l2} F(e^{-t}u+\sqrt{1-e^{-2t}}\, v)\d
    m_\beta(v),
  \end{align*}
  where the integral is a Bochner integral.
\end{defn}
The following properties are well known.
\begin{lemma}
  \label{thm_gaussapp10:5}
  The semi-group $P^\beta$ is ergodic in the sense that for any $u\in
  \l2$,
  \begin{equation*}
    P_t^\beta F(u)\xrightarrow{t\to \infty} \int f\d m_\beta.
  \end{equation*}
  Moreover, if $F$ belongs to $\CC_b^k(\l2;\, X)$, then, $\nabla^{(k)}
  (P_t^\beta F)=\exp(-kt)P_t^\beta (\nabla^{(k)}F)$ so that we have
  \begin{equation*}
    \int_0^\infty  \sup_{u\in \l2} \|\nabla^{(k)}
    (P_t^\beta F)(u)\|_{(\l2)^{\otimes (k)}\otimes X} \d t
    \le \frac 1k \|F\|_{\CC^k_b(\l2;\, X)}.
  \end{equation*}
\end{lemma}


We recall that for $X$ an Hilbert space and $A$ a linear continuous
map from $X$ into itself, $A$ is said to be trace-class whenever the
series $\sum_{n\ge 1} |(Af_n,\, f_n)_X|$ is convergent for one (hence
any) complete orthonormal basis $(f_n,\,n\ge 1)$ of $X$. When $A$ is
trace-class, its trace is defined as $\trace(A)=\sum_{n\ge 1} (Af_n,\,
f_n)_X$. For $x,\, y\in X$, the operator $x\otimes y$ can be seen
either as an element of $X\otimes X$ or as a continuous map from $X$
into itself via the identification : $x\otimes y(f)=(y,\, f)_X\,
x$. It is thus straightforward that such an operator is trace-class
and that $\trace(x\otimes y)=\sum_{n\ge 1}(y,\, f_n)_X(x,\,
f_n)_X=(x,\, y)_X$ according to the Parseval formula.  We also need to
introduce the notion of partial trace. For any vector space $X$,
$\Lin(X)$ is the set of linear operator from into itself. For $X$ and
$Y$ two Hilbert spaces, the partial trace operator along $X$ can be
defined as follows: it is the unique linear operator
\begin{equation*}
  \trace_X \, :\, \Lin(X\otimes Y) \longrightarrow \Lin(Y)
\end{equation*}
such that for any $R\in \Lin(Y)$, for any trace class operator $S$ on
$X$,
\begin{equation*}
  \trace_X(S\otimes R)=\trace_X(S)\, R.
\end{equation*}
For Hilbert valued functions, we define $A^\beta$ as follows.
\begin{defn}
  \label{lemma:generateur}
  Let $A^\beta$ denote the linear operator defined for $F\in
  \CC^2_b(\l2;\, X)$ by:
  \begin{equation*}
    (A^\beta F) (u)=u.(\nabla F)(u)-\trace_{\l2}(S_\beta \nabla^2 F (u)),\ \text{for all } u \in \l2.
  \end{equation*}
  We still denote by $A^\beta$ the unique extension of $A^\beta$ to
  its maximal domain.
\end{defn}
The map $A^\beta$ is the infinitesimal generator of $P^\beta$ in the
sense that for $F\in \CC^2_b(\l2;\, X)$: for any $u\in \l2$,
\begin{equation}\label{eq_gaussapp20:6}
  P^\beta_t F(u)=F(u)+\int_0^t A^\beta P^\beta_sF(u)\d s.
\end{equation}
As a consequence of the ergodicity of $P^\beta$ and of
\eqref{eq_gaussapp20:6}, we have the Stein representation formula: For
any sufficiently integrable function $F\, :\, \l2\to \R$,
\begin{equation}\label{eq_edgeworth:5}
  \int_{\l2} F\d m_\beta - \int_{\l2} F\d \nu =  \int_{\l2}\int_0^\infty A^\beta P_t^\beta F(x) \d t\d \nu(x).
\end{equation}

\section{Normal approximation of Poisson processes}
\label{sec:norm-appr-poiss}
Consider the process
\begin{equation*}
  N_\lambda(t)=\frac 1{\sqrt{\lambda}}\left( N(t) - \lambda t
  \right)=\frac 1{\sqrt{\lambda}}\left(  \sum_{n\ge 1} \car_{[T_n,\, 1]}(t)- \lambda t
  \right)
\end{equation*}
where $(T_n,\, n\ge 1)$ are the jump times of $N$, a Poisson process
of intensity $\lambda$. It is well known that $N_\lambda$ converges in
distribution on $\mathfrak D$ (the space of cadlag functions) to a
Brownian motion as $\lambda$ goes to infinity.  Since \begin{align*}
  I^\alpha_{0^+}\Bigl((.-\tau)_+^{-\beta}\Bigr)&=\Gamma(\alpha)^{-1}\int_\tau^t(s-\tau)^{-\beta}(t-s)^{\alpha-1}\d
  s\\
  &=\Gamma(1-\beta) \ (t-\tau)_+^{\alpha-\beta},
\end{align*}
we have $
I^\beta_{0^+}((.-\tau)_+^{-\beta})=\Gamma(1-\beta)\car_{[\tau,
  1]}(t)$. Hence, the sample-paths of $N_\lambda$ belong to
$\I_{\beta,2}$ for any $\beta<1/2$.  It also follows that
\begin{equation*}
  \embed N_\lambda=\sum_{n\ge 1} \frac{1}{\sqrt{\lambda}}\int_0^1
  k_n^{1-\beta}(s) (\d N(s)-\lambda \d s)\ x_n,
\end{equation*}
where, for any integer $n$,
\begin{equation*}
  k_n^{1-\beta}(t)=\frac{1}{\Gamma(1-\beta)}\int_t^1
  (s-t)^{-\beta}e_n(s)\d s.
\end{equation*}
For the sake of notations, we introduce
\begin{equation*}
  K_\lambda= \frac{1}{\sqrt{\lambda}}\sum_{n\ge 1}  \
  k_n^{1-\beta}\otimes x_n=\frac{1}{\sqrt{\lambda}}\, K_1.
\end{equation*}
The following theorem has been established in \cite{CD:2012}.
\begin{theorem}
  \label{thm_gaussapp10:2}
  For any $\lambda >0$, for any $F\in\CC^3_b(\l2;\, \R)$,
  \begin{equation}
    \label{eq_edgeworth:2}
    \left|    \esp{}{F(N_\lambda)}-\int_{\l2} F\d m_\beta\right| \le \frac{ 1}{6\sqrt{\lambda}}\ 
    c_{1-\beta}^3 \ \|F\|_{\CC^3_b(\l2;\, \R)}.
  \end{equation}
\end{theorem}
The proofs uses a few basic notions of Malliavin calculus with respect
to the Poisson process $N$ which we recall rapidly now (for details,
see \cite{CD:2012,MR2531026}). It is customary to define the discrete
gradient as
\begin{equation*}
  D_\tau F(N)=F(N+\delta_\tau)- F(N), \text{ for any } \tau \in
  [0,\, 1],
\end{equation*}
where $N+\delta_\tau$ is the point process $N$ with an extra atom at
time $\tau$. We denote by $\D_{2,1}$ the set of square integrable
functionals $F$ such that $\esp{}{\int_0^1 |D_\tau F(N)|^2\d\tau}$ is
finite.  We then have the following relationship:
\begin{equation}
  \label{eq_gaussapp6:4} \esp{}{F(N)\int_0^1 g(\tau) (\d
    N(\tau)-\lambda \d \tau)}=\lambda\ \esp{}{\int_0^1
    D_\tau F(N)\  g(\tau)\d \tau},
\end{equation}
for any $g\in \L^2([0,1])$ and any $F\in \D_{2,1}$. Moreover,
\begin{equation*}
  D_\tau\left(\int_0^1 g(s)(\d N(s)-\lambda \d s)\right)=g(\tau).
\end{equation*}
so that
\begin{math}
  D_\tau \embed N_\lambda=K_\lambda(\tau).
\end{math}
In what follows, we make the convention that a sum like
$\sum_{r=1}^0\ldots$ is zero. For any integers $r\ge k\ge 1$, consider
$\T_r^k$ the set of all $k$-tuples (ordered lists of length $k$) of
integers $(a_1,\cdots,\, a_k)$ such that $\sum_{i=1}^k a_i=r$ and
$a_i\ge 1$ for any $i\in \{1,\, \cdots,\,
k\}$. 
We denote by $()$ the empty list and for any $r$, $\T_r^0=\{()\}$ and
$\T_r=\cup_{j=1}^r \T_r^j$. For $a\in \T_r$, we denote by $|a|$ its
length, i.e. the unique index $j$ (necessarily less than $r$) such
that $a\in \T_r^j$.  For two tuples $a=(a_1,\cdots,a_k)$ and
$b=(b_1,\cdots,\, b_n)$, their concatenation $a\oplus b$ is the
$(k+n)$-tuple $(a_1,\cdots,a_n,b_1,\cdots,b_n)$. We define by
induction the following constants :
\begin{equation*}
  \Xi_{()}=1,\   \Xi_{(j)}=\frac{1}{(j+2)!} \text{ and } \Xi_{a\oplus
    (j)}=\frac{\Xi_a}{(j+1)!(r+j+2k+2)} \text{ for any }j\in \N, a\in \T_r^k.
\end{equation*}
For instance, we have
\begin{equation*}
  \Xi_{(1)}=\frac{1}{6},\ \Xi_{(2)}=\frac{1}{24},\ \Xi_{(1,1)}=\frac{1}{72}\cdotp
\end{equation*}
For any tuple $a=(a_1,\cdots,a_k)\in \T_r$, we set
\begin{equation*}
  \hatH^a=\bigotimes_{l=1}^k \int_0^1 K_1^{\otimes (a_l+2)}(\tau)\d\tau\in (\l2)^{\otimes (r+2|a|)}.
\end{equation*}
Consider also the sequence $(\xi_s,\, s\ge 0)$ given by the recursion
formula:
\begin{equation}
  \label{eq_gaussapp10:5}
  \xi_s=  \sum_{j=2}^{s+1}\frac{1}{(j+1)!}\
  \frac{\xi_{s+1-j}}{3s+7-2j}\, c_{1-\beta}^{j+1}+ \frac{\eta_{s+3}}{(s+3)!},
\end{equation}
and 
\begin{equation*}
  \eta_s=\frac{2}{2+s(1-2\beta)}\, \frac{1}{(1-2\beta)^{s/2}\Gamma(1-\beta)^{s/2}}\cdotp
\end{equation*}
\begin{theorem}
  \label{thm_gaussapp10:7}
  Let $s$ be a non negative integer. We denote by $\nu_\lambda^*$ the
  distribution of $\embed N_\lambda$ on $\l2$. For $F\in
  \CC^{3s+3}_b(\l2;\, \R)$, we have
  \begin{multline}\label{eq:3}
    \int_{\l2}F(u)\d \nu^*_\lambda(u)=\int_{\l2} F(u)\d
    m_\beta(u)\\
    +\sum_{r=1}^s \lambda^{-r/2} \sum_{a\in \T_r}\Xi_a\int_{\l2}
    \nabla^{(r+2|a|)}F(u).\hatH^{a}\d m_\beta(u) +\Rem(s,\, F,\,
    \lambda)
  \end{multline}
  where the remainder term can be bounded as
  \begin{equation*}
    \left| \Rem(s,\, F,\, \lambda)\right|\le \xi_{s} \ \lambda^{-(s+1)/2}\|F\|_{\CC^{3s+3}_b(\l2;\, \R)}.
  \end{equation*}
\end{theorem}
For $s=0$, this means that
\begin{equation*}
  \int_{\l2}F(u)\d \nu^*_\lambda(u)=\int_{\l2} F(u)\d
  m_\beta(u) + \Rem(0,\, F,\, \lambda)
\end{equation*}
where the remainder is bounded by
$\lambda^{-1/2}c_{1-\beta}^3\,\|F\|_{\CC^{3}_b(\l2;\, \R)}/6.$ That is
to say that it is the exact content of \thmref{thm_gaussapp10:2}. For
$s=1$, we obtain
\begin{multline*}
  \int_{\l2}F(u)\d \nu^*_\lambda(u)=\int_{\l2} F(u)\d m_\beta(u) \\ +
  \frac{\lambda^{-1/2}}{6} \int_{\l2} \nabla^{(3)}F(u).\hatH^{(1)}\d
  m_\beta(u)+\Rem(1,\, F,\, \lambda)
\end{multline*}
where $\Rem(1,\, F,\, \lambda)=O(\lambda^{-1})$ and for $s=2$, we get
\begin{multline*}
  \int_{\l2}F(u)\d \nu^*_\lambda(u)=\int_{\l2} F(u)\d m_\beta(u) +
  \frac{\lambda^{-1/2}}{6} \int_{\l2}
  \nabla^{(3)}F(u).\hatH^{(1)}\d m_\beta(u)\\
  +\lambda^{-1}\left[ \frac{1}{72}\int_{\l2}
    \nabla^{(6)}F(u).\hatH^{(1,1)}\d m_\beta(u)+\frac{1}{24}\int_{\l2}
    \nabla^{(4)}F(u).\hatH^{(2)}\d m_\beta(u)\right]\\
  +\Rem(2,\, F,\, \lambda),
\end{multline*}
with $\Rem(2,\, F,\, \lambda)=O(\lambda^{-3/2})$.
\begin{proof}
  As said before, for $s=0$, the proof reduces to that of
  \thmref{thm_gaussapp10:2}. Let $s\ge 1$ and assume that \eqref{eq:3}
  holds up to rank $s-1$.  Let $F\in \CC^{3s+2}_b(\l2;\R)$ and $x\in
  \l2$. Denoting by $G(y)=F(y)x$ for $y\in \l2$, we have
  \begin{align*}
    \esp{}{\embed N_\lambda. G(\embed
      N_\lambda)}&=\frac{1}{\sqrt{\lambda}}\sum_{n\ge 1}
    \esp{}{\int_0^1
      k_n^{1-\beta}(s) (\d N(s)-\lambda \d s)\ F(\embed N_\lambda)}\ x_n.x\\
    &=\frac{1}{\sqrt{\lambda}}\sum_{n\ge 1} \esp{}{\int_0^1
      k_n^{1-\beta}(\tau) D_\tau F(\embed N_\lambda)\lambda
      \d \tau}\ x_n.x\\
    &=\sqrt{\lambda}\ \esp{}{\int_0^1 D_\tau F(\embed
      N_\lambda). K_1(\tau)\d\tau}.
  \end{align*}
  According to the Taylor formula at order $s+1$,
  \begin{multline*}
    D_\tau F(\embed N_\lambda) = F(\embed N_\lambda+K_\lambda(\tau))-F(\embed N_\lambda)\\
    \shoveleft{ =\sum_{j=1}^{s+1}\frac{\lambda^{-j/2}}{j!}\nabla^{(j)}
      F(\embed
      N_\lambda).K_1(\tau)^{\otimes j}} \\
    +\frac{\lambda^{-(s+2)/2}}{(s+1)!}\int_0^1 (1-r)^{s+1}\
    \nabla^{(s+2)} F(\embed N_\lambda + r
    K_\lambda(\tau)). K_1(\tau)^{\otimes (s+2)} \d r.
  \end{multline*}
  Thus, we get
  \begin{multline}\label{eq_edgeworth:4}
    \esp{}{\embed N_\lambda.G(\embed N_\lambda)}
    =\sum_{j=1}^{s+1}\frac{\lambda^{-(j+1)/2}}{j!}\int_0^1
    \esp{}{\nabla^{(j)} F(\embed
      N_\lambda).K_1(\tau)^{\otimes (j+1)} }\d\tau\\
    +\frac{\lambda^{-(s+1)/2}}{(s+1)!}\int_0^1\int_0^1 (1-r)^{s+1}\
    \esp{}{ \nabla^{(s+2)} F(\embed N_\lambda + r
      K_\lambda(\tau)). K_1(\tau)^{\otimes (s+3)}} \d r\d\tau.
  \end{multline}
  By linearity and density, \eqref{eq_edgeworth:4} holds for any $G\in
  \C^{3s+2}_b(\l2;\l2).$ According to the Stein representation formula
  \eqref{eq_edgeworth:5}, we get:
  \begin{multline*}
    \esp{}{F(N_\lambda)}=\int_{\l2}F(u)\d m_\beta(u)\\
    + \sum_{j=2}^{s+1} \frac{\lambda^{-(j-1)/2}}{j!}  \int_0^\infty
    \int_{\l2}
    \nabla^{(j+1)}P_t^\beta F(u).(\int_0^1 K_1(\tau)^{\otimes  (j+1)}\d \tau) \d \nu^*_\lambda(u)\d t\\
    \shoveleft{ +\frac{\lambda^{-(s+1)/2}}{(s+1)!}\times }\\
    \int_{\l2}\int_0^\infty \int_0^1\int_0^1
    (1-\theta)^{s+1}\nabla^{(s+3)}P_tF(u+\theta
    K_1(\tau)).K_1(\tau)^{\otimes (s+3)}\d
    \theta \d \tau\d t \d \nu^*_\lambda(u)\\
    =\int_{\l2}F(u)\d m_\beta(u) + A_1+ A_2.
  \end{multline*}
  For any $j$, we apply the recursion hypothesis of rank $s+1-j$ to
  the functional
$$F_j: u\mapsto \int_0^\infty \nabla^{(j+1)}P_t^\beta F(u).\hatH^{(j-1)} \d t.$$ Thus, we have:
\begin{multline}\label{eq:4}
  A_1=\\
\shoveleft{ \sum_{j=2}^{s+1} \frac{\lambda^{-(j-1)/2}}{j!}
  \sum_{r=0}^{s-j+1}\lambda^{-r/2} \sum_{a\in \T_r}\Xi_a
  \int_0^\infty\int_{\l2} \nabla^{(2|a|+r+j+1)}P_t^\beta F(u).\hatH^{a\oplus (j-1)}\d m_\beta(u)}\d t\\
  + \sum_{j=2}^{s+1} \frac{\lambda^{-(j-1)/2}}{j!}  \Rem(s+1-j,\,
  F_j ,\, \lambda)\vphantom{\int_{\l2}\int_0^1}=B_1+B_2
\end{multline}
According to the commutation relationship between $\nabla$ and
$P^\beta$ and since $m_\beta$ is $P^\beta_t$-invariant, it follows
that
\begin{multline*}
  B_1= \sum_{j=2}^{s+1} \sum_{r=0}^{s-j+1}\frac{\lambda^{-(r+j-1)/2}}{j!}\\
  \times \sum_{a\in \T_r}\frac{\Xi_a}{2|a|+r+j+1} \int_{\l2}
  \nabla^{(2|a|+r+j+1)} F(u).\hatH^{a}\otimes \hatH^{(j-1)}\d
  m_\beta(u).
\end{multline*}
We now proceed to the change of variables $r\leftarrow r+j-1$,
$j\leftarrow j-1$ so that we have
\begin{multline*}
  B_1= \sum_{r=1}^{s}\lambda^{-r/2} \sum_{j=1}^{r}\\
  \times \sum_{a\in \T_{r-j}}\frac{\Xi_a}{(j+1)!(2|a|+r+2)} \int_{\l2}
  \nabla^{(2|a|+r+2)} F(u).\hatH^{a}\otimes \hatH^{(j)}\d m_\beta(u).
\end{multline*}
If $a$ belongs to $\T_{r-j}$ then $b=a\oplus(j)$ belongs to $\T_r$ and
$2|a|+r+2=2|b|+2$. In the reverse direction, for $b\in \T_r$, by
construction, the last component (say on the right) of $b$ belongs to
$\{1,\cdots,r\}$. Let $j$ denote the value of this component, it
uniquely determines $a$ such that $b=a\oplus (j)$ with $a\in \T_{r-j}$
thus $\cup_{j=1}^r \T_{r-j}=\T_r$ (where the union is a disjoint
union) and $B_1$ can be written as
\begin{equation*}
  B_1= \sum_{r=1}^{s}\lambda^{-r/2}  \sum_{a\in \T_{r}}\Xi_a \int_{\l2} \nabla^{(r+2|a|)} F(u).\hatH^{a}\d m_\beta(u).
\end{equation*}
Now, we estimate the remainder at rank $s$. According to the previous
expansions,
\begin{multline*}
  \Rem(s,\, F,\, \lambda)= A_2+B_2 =\sum_{j=2}^{s+1}
  \frac{\lambda^{-(j-1)/2}}{(j+1)!}  \Rem(s+1-j,\, F_j ,\,
  \lambda) \\
  +\frac{\lambda^{-(s+1)/2}}{(s+1)!}  \int_{\l2}\int_0^\infty \int_0^1
  (1-\theta)^{s+1}\nabla^{(s+3)}P_tF(u+\theta
  K_1(\tau)).K_1(\theta)^{(s+3)}\d  \theta \d t \d \nu^*_\lambda(u)\\
  \shoveleft{\le \sum_{j=2}^{s+1}
    \frac{\lambda^{-(j-1)/2}}{(j+1)!}\xi_{s+1-j} \lambda^{-(s+2-j)/2}
    \| F_j\|_{\CC^{3(s+1-j)+3}_b(\l2;\, \R)}
    +\frac{\lambda^{-(s+1)/2}}{(s+1)!}\times}\\
  \int_{\l2}\int_0^\infty e^{-(s+3)t}\int_0^1
  (1-\theta)^{s+1}P_t\nabla^{(s+3)}F(u+\theta
  K_1(\tau)).K_1(\theta)^{(s+3)}\d \theta \d t \d \nu^*_\lambda(u)
\end{multline*}
Since $\nabla^{(s+3)}F$ is bounded and $P^\beta$ is a Markovian
semi-group, for any $t\ge 0$, $P^\beta_t\nabla^{(s+3)}F$ is bounded by
$\|F\|_{\CC^{s+3}_b(\l2;\, \R)}$. Hence,
\begin{multline*}
  \Rem(s,\, F,\, \lambda)\le \lambda^{-(s+1)/2}\sum_{j=2}^{s+1}
  \frac{\xi_{s+1-j}}{(j+1)!} \| F\|_{\CC^{3s+6-3j}_b(\l2;\, \R)}
  \\ \shoveright{+ \frac{\lambda^{-(s+1)/2}}{(s+3)!}\ c_{1-\beta}^{s+1}\|F\|_{\CC^{s+3}_b(\l2;\, \R)}}\\
  \shoveleft{\le \lambda^{-(s+1)/2}\sum_{j=2}^{s+1}  \frac{\xi_{s+1-j}\, c_{1-\beta}^{j+1}}{(j+1)!\ (3s+7-2j)}\  \| F\|_{\CC^{3s+7-2j}_b(\l2;\, \R)}}\\
  +
  \frac{\lambda^{-(s+1)/2}}{(s+3)!}c_{1-\beta}^{s+1}\|F\|_{\CC^{s+3}_b(\l2;\,
    \R)}.
\end{multline*}
Since $\sup\limits_{j=2,\cdots,s+1} 3s+7-2j=3s+3$ and $s+3\le 3s+3$,
the result follows.
\end{proof}

\section{Linear interpolation of the Brownian motion}
\label{sec:linear-interpolation}

For $m\ge 1$, the linear interpolation $B_m^\dag$ of a Brownian motion
$B^\dag$ is defined by
\begin{equation*}
  B_m^\dag(0)=0 \text{ and }  \d B_m^\dag(t)=m\sum_{i=0}^{m-1} (B^\dag(i+1/m)-B^\dag(i/m)) \car_{[i/m,\, (i+1)/m)}(t)\d t.
\end{equation*}
Thus, $\embed B_m^\dag$ is given by
\begin{equation*}
  \embed B_m^\dag=\left( m\sum_{i=0}^{m-1} (B^\dag(i+1/m)-B^\dag(i/m)) \int_{i/m}^{(i+1)/m}k_n^{1-\beta}(t)\d t, \ n\ge 1\right).
\end{equation*}
Consider the $\L^2([0,\, 1])$-orthonormal functions
\begin{equation*}
  e^m_j(s)=\sqrt{m}\, \car_{[j/m,\, (j+1)/m)}(s),\, j=0,\, \cdots,\, m-1,\,  s\in [0,\, 1]
\end{equation*}
and $F_m^\dag=\text{span}(e^m_j,\, j=0,\, \cdots,\, m-1).$ We denote
by $p_{F_m^\dag}$ the orthogonal projection over $F_m^\dag$.  Since
$B_m^\dag$ is constructed as a function of a standard Brownian motion,
we work on the canonical Wiener space $({\mathcal C}^0([0,\, 1];\,
\R),\ \I_{1,\, 2},\, m^\dag)$. The gradient we consider, $D^\dag$, is
the derivative of the usual gradient on the Wiener space and the
integration by parts formula reads as:
\begin{equation}
  \label{eq_gaussapp10:10}
  \esp{m^\dag}{F \int_0^1 u(s)\d B^\dag(s)}=\esp{m^\dag}{\int_0^1
    D^\dag_s F\  u(s)\d s}
\end{equation}
for any $u\in \L^2([0,\, 1])$.
We need to introduce some constants which already appeared in \cite{CD:2012}.
For any $\alpha\in (0,1]$, let 
\begin{equation*}
  d_\alpha=\max\left(\sup_{z\ge 0}\int_0^{z}s^{\alpha-1}\cos(\pi s)\d s,\ \sup_{z\ge 0}\int_0^{z}s^{\alpha-1}\sin(\pi s)\d s\right).
\end{equation*}
Moreover, 
\begin{theorem}[cf. \cite{CD:2012}]
  \label{thm_gaussapp10:8} Let $\nu_m^\dag$ be the law of $\embed
  B_m^\dag$ on $\l2$ and let
  \begin{equation*}
    H^\dag_m=(p_{F_m^\dag}k_n^{1-\beta},\, n\ge 1).
  \end{equation*}
  For any $F\in \CC^2_b(\l2;\, \R)$,
  \begin{equation*}
    \left|    \int_{\l2} F\d\nu_m^\dag -\int_{\l2} F\d m_\beta\right|\le
    \frac{\gamma^\dag_{m,\,\beta}}{2}\, \|F\|_{\CC^2_b(\l2;\, \R)},
  \end{equation*}
  where, for any $0<\varepsilon <1/2-\beta$, for any $p$ such that
  $p(1/2-\beta-\varepsilon)>1$,
  \begin{equation*}
    \gamma^\dag_{m,\, \beta}\le \frac{d_{1/2+\varepsilon}c_{1-\beta} \zeta_{1/2+\varepsilon,\, p}}{\Gamma(1/2+\varepsilon)} \left(\sum_{n\ge 1} \frac{1}{n^{1+2\varepsilon}}\right)^{1/2} m^{-(1/2-\beta-\varepsilon)},  
  \end{equation*}
  with for any $p\ge 1$, $\alpha \in (1/p,\,1],$
  \begin{equation*}
    \zeta_{\alpha,\, p}=  \sup_{\|f\|_{\I_{\alpha,\, p}}=1}\|f\|_{\Hol_0(\alpha-1/p)}.
  \end{equation*}
\end{theorem}
\begin{theorem}
  \label{thm_gaussapp10:9}
  For any integer $s$, for any $F\in {\mathcal C}^{2s+2}_b(\l2;\,
  \R)$, we have the following expansion:
  \begin{multline*}
    \esp{\nu^\dag_m}{F}= \sum_{j=0}^s \frac{1}{2^j \, j!}
    \int_{\l2}\langle \nabla^{(2j)}F(u),\ (S_m^\dag-S_\beta)^{\otimes
      j}\rangle_{(\l2)^{\otimes 2j}}\d
    m_\beta(u)\\
    + \Rem^\dag(s,\, F,\, m),
  \end{multline*}
  where $S_m^\dag=\trace_{\L^2([0,\, 1])}(K_m^\dag\otimes K_m^\dag)$
  and $\Rem ^\dag(s,\, F,\,m)$ can be bounded by
  \begin{equation*}
    \left|\Rem^\dag(s,\, F,\, m)\right| \le \frac{(\gamma^\dag_{m,\beta})^{s+1}}{2^{s+1}}\ \|F\|_{{\mathcal C}^{2s+2}_b(\l2;\,
      \R)}\cdotp
  \end{equation*}
\end{theorem}
\begin{proof}
  For $s=0$, the result boils down to \thmref{thm_gaussapp10:8}. We
  proceed by induction on $s$. According to the induction hypothesis and to the Stein representation formula \cite{CD:2012},
  for $F$ sufficiently regular,
  \begin{multline*}
    \int_{\l2} F(u)\d \nu_m^\dag(u) - \int_{\l2} F(u) \d m_\beta(u) \\
    \shoveleft{= \int_{\l2} \int_0^\infty \nabla^{(2)}P_t^\beta
      F(u).\, (S_m^\dag-S_\beta)\d t\d \nu_m^\dag
      (u)}\\
    \shoveleft{=\sum_{j=0}^s \frac{1}{2^j \, j!}  \int_{\l2}
      \nabla^{(2)}(\int_0^\infty \nabla^{(2j)}P_t^\beta F(u).\
      (S_m^\dag-S_\beta)^{\otimes ( j)}\d t).(S_m^\dag-S_\beta)\d
      m_\beta(u)}\\
    \shoveright{+ \Rem^\dag(s,\, \int_0^\infty
      \nabla^{(2)}P_t^\beta F.(S_m^\dag-S_\beta) \d t,\, m)}\\
    \shoveleft{=\sum_{j=0}^s \frac{1}{2^j \, j!}  \int_{\l2}
      \int_0^\infty \nabla^{(2j+2)}P_t^\beta F(u).\
      (S_m^\dag-S_\beta)^{\otimes ( j+1)} \d t \d
      m_\beta(u)}\\
    + \Rem^\dag(s,\, \int_0^\infty \nabla^{(2)}P_t^\beta
    F.(S_m^\dag-S_\beta) \d t,\, m).
  \end{multline*}
  Since $\nabla^{(j)}P_t^\beta F(u)=e^{-jt}P_t^\beta \nabla^{(j)}F(u)$
  and since $m_\beta$ is invariant under the action of $P^\beta $, we
  obtain
  \begin{multline*}
    \int_{\l2} F(u)\d \nu_m^\dag(u) = \int_{\l2} F(u) \d m_\beta(u) \\
    + \sum_{j=0}^{s} \frac{1}{2^{j+1}}  \int_{\l2}
    \nabla^{(2j+2)}F(u).\ (S_m^\dag-S_\beta)^{\otimes (j+1)}\d
    m_\beta(u)\\
    + \Rem^\dag(s,\, \int_0^\infty \nabla^{(2)}P_t^\beta F.\
    (S_m^\dag-S_\beta)\d t,\, m).
  \end{multline*}
  By a change of index in the sum, we obtain the main part of the
  expansion for the rank $s+1$. Moreover,
  \begin{multline*}
    |  \Rem^\dag(s+1, \, F,\, m)|\le \frac{(\gamma_{m,\, \beta}^\dag)^s}{2^{s+1} (s+1)!}\ \|\int_0^\infty \nabla^{(2)}P_t^\beta  F.(S_m^\dag-S_\beta)\d t \|_{{\mathcal C}^{2(s+1)}_b(\l2)}\\
    \le \frac{(\gamma_{m,\, \beta}^\dag)^s}{2^{s+1}}
    \frac{\|S_m^\dag-S_\beta\|_{\l2}}{2}\ \|F\|_{{\mathcal
        C}^{2(s+2)}_b(\l2)}.\end{multline*} Since the norm of a
  bounded operator is bounded by its trace provided that the latter
  exists, we have $\|S_m^\dag-S_\beta\|_{\l2}\le \gamma_{m,\,
    \beta}^\dag$, hence the result.
\end{proof}


\end{document}